\documentclass[10pt]{article}

\pdfoutput=1 %ArXiv PDFLATEX
\usepackage{etex}
\usepackage[french, english]{babel} %André : ajout de «french» pour permettre l'utilisation de guillemets dans le fichier .bib. TOUJOURS laisser english à la fin (main language).
\usepackage[utf8]{inputenc}
\usepackage[T1]{fontenc}
\usepackage{lmodern}
\usepackage{authblk}
\usepackage{wrapfig}

\usepackage{amsmath, latexsym, amssymb, amsfonts, amsthm}
\numberwithin{equation}{section}%    Permet la numérotation des équations en fonction des sections
\usepackage{mathtools}% Ajout André. Nécessaire pour l'environnement {multlined}.
\usepackage{graphicx}
\usepackage{subcaption}
\usepackage{enumitem}
\usepackage{multicol}
\usepackage[colorlinks,citecolor=black,filecolor=black,linkcolor=black,urlcolor=black ]{hyperref}
\def\al{\alpha}
\def\b{\beta}

\def\de{\delta}

\def\be{\begin{equation}}   \def\ee{\end{equation}}
\def\ba   {\begin{array}}      \def\ea   {\end{array}}
\def\bea  {\begin{eqnarray}}   \def\eea  {\end{eqnarray}}
\def\bean {\begin{eqnarray*}}  \def\eean {\end{eqnarray*}}

\newtheorem{theorem} {Theorem}

\newtheorem{definition} {Definition}

\newtheorem{corollary} {Corollary}

\newtheorem{conjecture} {Conjecture}
\newtheorem{proposition} {Proposition}
\theoremstyle{definition}

\newcommand{\mR}{\ensuremath{\mathbb{R}}}
\newcommand{\mT}{\ensuremath{\mathbb{T}}}

\newcommand{\BC}{\ensuremath{\mathbb{B}\mathbb{C}}}

\newcommand{\im}[1]{\ensuremath{\mathrm{\mathbf{i}}_{\mathbf{#1}}}}

\newcommand{\TC}{\ensuremath{\mathbb{TC}}}

\newcommand{\jh}[1]{\ensuremath{\mathrm{\mathbf{j}}_{\mathbf{#1}}}}
\newcommand{\M}[1]{\ensuremath{\mathcal{M}_{#1}}}

\newcommand{\T}{\ensuremath{\mathcal{T}}}
\newcommand{\Tbb}{\mathbb{T}}
\newcommand{\Te}{\ensuremath{\T_{e}}}
\newcommand{\Ta}{\T^{\ast}}
\newcommand{\N}{\ensuremath{\mathbb{N}}}
\newcommand{\R}{\ensuremath{\mathbb{R}}}
\newcommand{\C}{\ensuremath{\mathbb{C}}}
\newcommand{\Co}{\ensuremath{\mathbb{C}}}
\newcommand{\D}{\mathbb{D}}
\newcommand{\Mu}[1]{\mathbb{M}({#1})}
\newcommand{\Hy}{\ensuremath{\mathcal{H}}}
\newcommand{\Hp}{\ensuremath{\mathcal{H}'}}
\newcommand{\suite}{\{P_{c}^{(n)}(0)\}_{n\in\mathbb{N}}}
\newcommand{\un}{\ensuremath{\gamma_{1}}}
\newcommand{\uc}{\ensuremath{\overline{\gamma_{1}}}}
\newcommand{\dx}{\gamma_{2}}
\newcommand{\dc}{\overline{\gamma_{2}}}
\newcommand{\tr}{\ensuremath{\gamma_{3}}}
\newcommand{\tc}{\ensuremath{\overline{\gamma_{3}}}}
\newcommand{\ut}{\ensuremath{\gamma_{1}\gamma_{3}}}
\newcommand{\uct}{\ensuremath{\overline{\gamma_{1}}\gamma_{3}}}
\newcommand{\utc}{\ensuremath{\gamma_{1}\overline{\gamma_{3}}}}
\newcommand{\uctc}{\ensuremath{\overline{\gamma_{1}}\overline{\gamma_{3}}}}
\newcommand{\vspan}{\text{span}_{\R}}

\newcommand*{\mcap}{\mathbin{\raisebox{-.2ex}{\scalebox{1.3}{\ensuremath{\cap}}}}}

\newcommand{\Mnp}{\ensuremath{\mathcal{M}_n^p}}

\usepackage{todonotes}

\title{Relationship between the Mandelbrot Algorithm and the Platonic Solids}
\author{André Vallières\thanks{E-mail: {\tt andre.vallieres@uqtr.ca}} }
\author{Dominic Rochon\thanks{E-mail: {\tt dominic.rochon@uqtr.ca}}}
\affil{Département de mathématiques et d'informatique, \\ Université du Québec à Trois-Rivières, C.P. 500, Trois-Rivières, Québec, Canada, G9A 5H7}
\date{\small{« Chaque flot est un ondin qui nage dans le courant, chaque courant est un sentier qui serpente vers mon palais, et mon palais est bâti fluide, au fond du lac, dans le triangle du feu, de la terre et de l’air. » \\ — Aloysius Bertrand, Gaspard de la nuit, 1842}}

\begin{document}
\maketitle

%%Abstract
\begin{abstract}
This paper focuses on the dynamics of the eight tridimensional principal slices of the tricomplex Mandelbrot set: the Tetrabrot, the Arrowheadbrot, the Mousebrot, the Turtlebrot, the Hourglassbrot,  the Metabrot, the Airbrot (octahedron) and the Firebrot (tetrahedron). In particular, we establish a geometrical classification of these 3D slices using the properties of some specific sets that correspond to projections of the bicomplex Mandelbrot set on various two-dimensional vector subspaces, and we prove that the Firebrot is a regular tetrahedron. Finally, we construct the so-called “Stella octangula” as a tricomplex dynamical system composed of the union of the Firebrot and its dual, and after defining the idempotent 3D slices of $\M3$, we show that one of them corresponds to a third Platonic solid: the cube.
\end{abstract}\vspace{0.5cm}

%
%we propose a general invertibility criterion for tricomplex numbers which unifies the known incomplete approaches, in addition to providing an explicit algebraic formula for the multiplicative inverse of a given tricomplex number based on their seven different conjugates
%
%
%
%
\newpage
\section*{Introduction}

Quadratic polynomials iterated on hypercomplex algebras have been used to generate multidimensional Mandelbrot sets for several years \cite{bedding, BrouilletteRochon, Dang, GarantRochon, Katunin, norton, Rochon1, Senn, Wang, Wang2}. Although this approach is widespread in the litterature, other attempts at generalizing the classic fractal to higher dimensions have been made \cite{Barrallo, Fishback, Garg}. While Bedding and Briggs \cite{bedding} established that possibly no interesting dynamics occur in the case of the quaternionic Mandelbrot set, the generalization given in \cite{Rochon1}, which uses the four-dimensional commutative algebra of bicomplex numbers, possesses an interesting fractal aspect reminiscent of the classical Mandelbrot set. Figure \ref{Deep_Tetrabrot} shows that phenomenon for the so-called Tetrabrot (Fig. \ref{fig:tetrabrot}). 

\begin{wrapfigure}{r}{0.51\textwidth}
\centering
	\begin{subfigure}{0.45\textwidth}
	\centering
	\includegraphics[width=5.5cm]{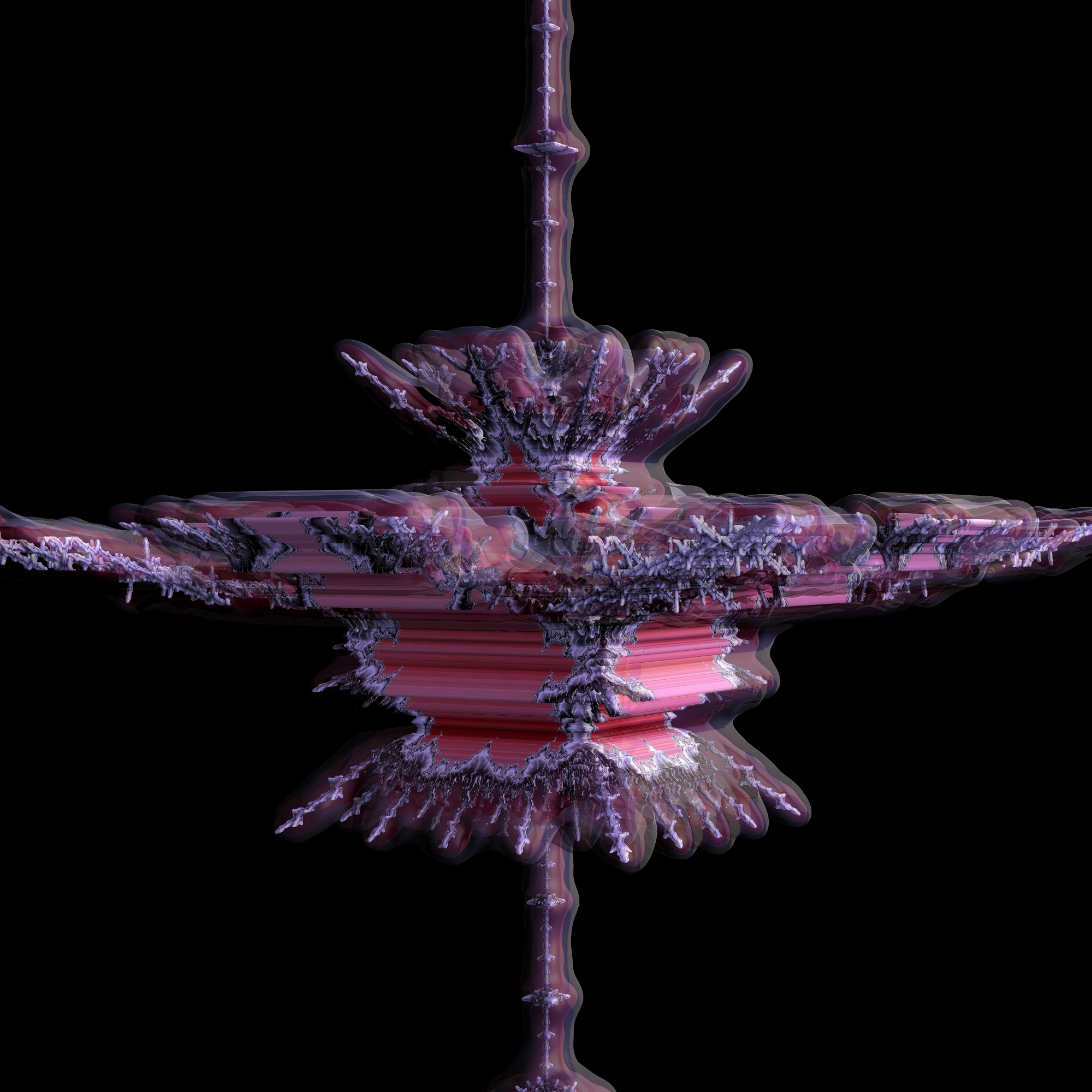}
	%\caption{$\T(1,\im1,\im2)$}
	\label{fig:Deep_Tetrabrot_01}
	\end{subfigure}
\vfill
	\begin{subfigure}{0.45\textwidth}
	\centering
	\includegraphics[width=5.5cm]{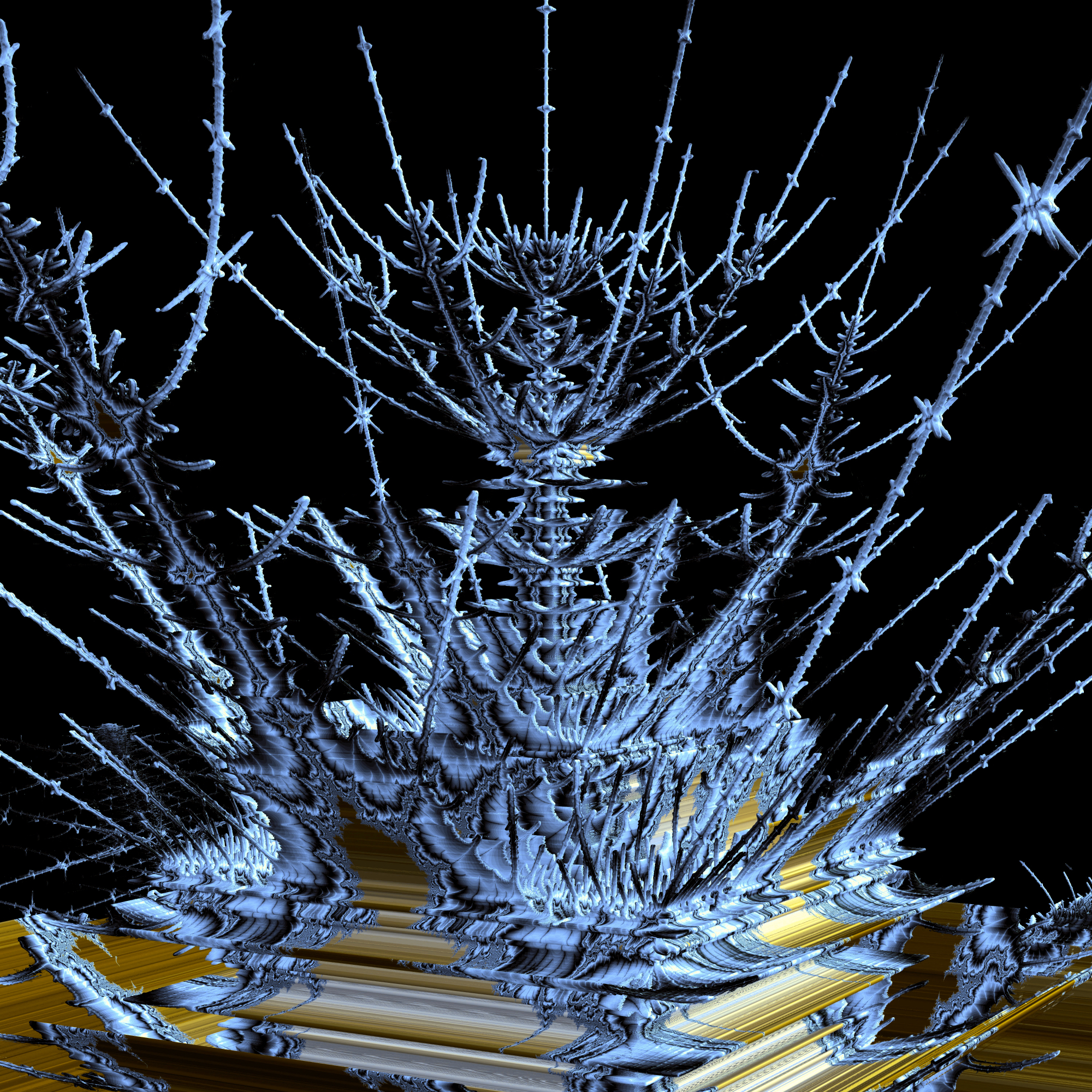}
	%\caption{$\T(1,\im1,\jh1)$}
	\label{fig:Deep_Tetrabrot_02}
	\end{subfigure}
\caption{Deep zoom\\ on the Tetrabrot.}
\label{Deep_Tetrabrot}
\end{wrapfigure}

This bicomplex Mandelbrot set $\M2$ is proven to be connected \cite{Rochon1} and related to a bicomplex version of the Fatou-Julia theorem \cite{Rochon2}. In 2009, these results and ideas were subsequently extended to the multicomplex space for quadratic polynomials of the form $z^2+c$ over multicomplex numbers \cite{GarantRochon}. In addition, the same authors introduced an equivalence relation between the fifty-six principal 3D slices of the tricomplex Mandelbrot set $\M3$ in order to establish which slices have the same dynamics and appearance in 3D visualization software. By doing this, eight equivalence classes were identified and characterized \cite{BrouilletteRochon}, and for each of those, one particular slice was designated as a canonical representative (Fig. \ref{fig:8 slices}).

Later, the scope of this comprehensive study was broadened by the use of the polynomial $z^p+c$ where $p\geq2$. In particular, the authors of \cite{RochonParise2,RochonRansfordParise} determined the exact intervals corresponding to $\mathcal{M}^{p}\cap\R$, depending on whether the integer $p$ is odd or even, and showed that the tricomplex Mandelbrot set $\mathcal{M}_{3}^{3}$ generated by the cubic polynomial $z^3+c$ has only four principal 3D slices \cite{RochonParise}. Then, Brouillette and Rochon \cite{Brouillette,BrouilletteRochon} generalized this result to the multicomplex space by establishing that the multicomplex Multibrot set $\Mnp$, $n\geq 2$, possesses at most nine distinct dynamics when $p$ is even, and at most four when $p$ is odd. In particular, the authors show that every principal 3D slice of $\Mnp$ is equivalent to a tricomplex slice up to an affine transformation, thus establishing the optimality of the tricomplex space in the context of principal 3D slices.

Nevertheless, there are a few aspects of the theory that have not yet been addressed. Indeed, although the principal 3D slices of the tricomplex Mandelbrot set have been classified according to their dynamics, the relationships between some of them remain largely unexplored. Moreover, surprisingly there are two principal 3D slices that exhibit no irregularity, which is in sharp contrast to the other six. In fact, one is a regular octahedron (the Airbrot, Fig.\ref{fig:airbrot}) while the other is conjectured to be a regular tetrahedron (the Firebrot, Fig.\ref{fig:firebrot}) \cite{GarantRochon}, and the underlying mechanism explaining such behavior is currently poorly understood.

Thus, the main objective of this paper is to deepen knowledge on the aspects mentioned above, and by extension, to establish a relationship between the Mandelbrot algorithm and the Platonic solids, which have become an integral part of natural sciences like chemistry and geology due to their remarkable properties and prevalence in those fields (see for example \cite{Daintith, Fragmentation, Tang}). To achieve this, we develop several geometrical characterizations for the principal 3D slices of $\M3$. However, in order to do so, we first need to recall relevant properties in the algebras of bicomplex and tricomplex numbers and introduce new ones.

Therefore, in \autoref{sec:tricomplex}, the algebras of bicomplex and tricomplex numbers are introduced and new results concerning the latter are presented. Then, we establish various geometrical characterizations for the principal 3D slices of $\M3$ in \autoref{sec:geometrical} and we also discuss some of the geometric properties that can be inferred from them, notably the fact that one slice is a regular tetrahedron. In \autoref{sec:cube}, we introduce another type of 3D projection called an \textit{idempotent tridimensional slice}, and we show that one of them is a cube. Finally, we talk about the possibility to generate, in the same 3D subspace as that of the Firebrot and its geometric dual, a regular compound called the stellated octahedron (also named the \textit{Stella octangula}) as a tricomplex dynamical system.

\section{The algebra of tricomplex numbers}\label{sec:tricomplex}
\subsection{Definitions and basics}
Bicomplex and tricomplex numbers are special cases of multicomplex numbers, which were first introduced by Segre \cite{Segre} in 1892. One important reference on the subject is due to Price \cite{Baley}, who studied the general multicomplex number system and provided details for the bicomplex case. A modern recursive definition for the set of multicomplex numbers of order $n$ would be \cite{BrouilletteRochon,GarantRochon,Baley}:
\begin{equation}\label{eq:multicomplex}
\Mu{n}:=\{\eta_1+\eta_2\im{n}:\eta_1,\eta_2\in\Mu{n-1}\}
\end{equation}
where $n\in\N$, $\im{n}^2=-1$ and $\Mu{0}:=\R$. Multicomplex addition and multiplication are defined in an analogous manner to that of the complex plane\footnote{Moreover, the set $\Mu{n}$ together with multicomplex addition and multiplication by real numbers is a vector space over the field $\R$ and is isomorphic to $\R^{2^{n}}$.}:
\begin{align*}
&(\eta_1+\eta_2\im{n}) + (\zeta_1+\zeta_2\im{n}) = (\eta_1+\zeta_1) + (\eta_2+\zeta_2)\im{n};\\
&(\eta_1+\eta_2\im{n})(\zeta_1+\zeta_2\im{n}) = (\eta_1\zeta_1 - \eta_2\zeta_2) + (\eta_1\zeta_2+\eta_2\zeta_1)\im{n}.
\end{align*}
Hence, we have $\Mu{1}\simeq\C(\im1)$, and setting $n=2$ gives the bicomplex numbers $\Mu{2}:=\BC$, which have been studied extensively \cite{Bicomplex,Baley,RochonMaitrise,Rochon3,RochonShapiro}.\footnote{It should be noted that beginning in 1848, J. Cockle developed an algebra which he called the \textit{tessarines} \cite{cockle1,cockle2,cockle3,cockle4} that was later proved to be isomorphic to $\BC$.} It follows immediately that a bicomplex number $w$ can be written as
\begin{gather*}
w=z_1+z_2\im2,\quad z_1,z_2\in\C(\im1)
\intertext{and expressing both $z_1$ and $z_2$ by their real coefficients $x_i$ yields}
w=x_1+x_2\im1+x_3\im2+x_4\jh1
\end{gather*}
where $\jh1=\im1\im2=\im2\im1$. Note that \jh1 is called a \textit{hyperbolic} unit since it satisfies the property $\jh1^2=1$ \cite{RochonShapiro,Sobczyk}. One direct but far-reaching consequence stemming from the existence of such a unit is the presence of non-trivial idempotent elements in $\BC$, namely $\un=\frac{1+\jh1}{2}$ and $\uc=\frac{1-\jh1}{2}$. As demonstrated in the references listed above, these two peculiar bicomplex numbers are the key to extend many classical results from the complex plane. In fact, more generally, the idempotent elements existing in $\Mu{n}$ generate various representations of a given multicomplex number, which in turn can be used to give natural extensions to concepts like holomorphy and power series \cite{Baley,RochonMaitrise,Holomorphy,vajiac}. In our context, the formal identity
\begin{equation}\label{eq:idemp BC}
w=(z_1-z_2\im1)\un+(z_1+z_2\im1)\uc
\end{equation}
which holds for all $w=z_1+z_2\im2\in\BC$ is called the idempotent representation of $w$ and will be used in \autoref{sec:geometrical}. Remark that in this form, multiplication can be carried out term by term.

We now turn our attention to the set of tricomplex numbers, which is denoted $\TC$ and corresponds to the case $n=3$ in \eqref{eq:multicomplex}:
\[
\TC=\Mu3 := \{\eta=\eta_{1}+\eta_{2}\im3:\eta_{1},\eta_{2}\in\BC,\im3^{2}=-1\}. 
\]
It follows from this definition that any tricomplex number can be expressed as a sum of two, four or eight terms, respectively having bicomplex, complex or real coefficients:
\begin{align*}
\eta &= \eta_{1}+\eta_{2}\im3 \\
&= z_{1}+z_{2}\im2 + z_{3}\im3+z_{4}\jh3 \\
&= x_{1} + x_{2}\im1 + x_{3}\im2 + x_{4}\jh1 + x_{5}\im3 + x_{6}\jh2 + x_{7}\jh3 + x_{8}\im4,
\end{align*}
where $\im4=\im1\im2\im3$ is a fourth imaginary unit and $\jh3=\im2\im3=\im3\im2,\,\jh2=\im1\im3=\im3\im1$ are other instances of hyperbolic units. Hence, there exists additional non-trivial idempotent elements in $\TC$, such as $\tr=\frac{1+\jh3}{2},\tc=\frac{1-\jh3}{2},\dx=\frac{1+\jh2}{2}$ and $\dc=\frac{1-\jh2}{2}$. Note that $\tr\tc=\dx\dc=0$ and $\tr+\tc=\dx+\dc=1$, which are the properties that make the set $\{\un,\uc\}$ interesting and useful in the setting of bicomplex numbers. In fact, it is possible to derive an idempotent representation in each of the sets $\{\tr,\tc\}$ and $\{\dx,\dc\}$ akin to that of bicomplex numbers. More will be said on the subject in \autoref{subsec:reps}.

Furthermore, the existence of idempotent elements that do not cancel out when multiplied together (e.g. $\un$ and $\tr$) along with the commutativity of multiplication in $\TC$ imply that the products $\un\tr,\uc\tr,\un\tc$ and $\uc\tc$ are other examples of idempotent elements. Note that the same could be said about the products $\gamma_{2}\gamma_{i},\overline{\gamma_{2}}\gamma_{i},\gamma_{2}\overline{\gamma_{i}}$ and $\overline{\gamma_{2}}\overline{\gamma_{i}}$ where $i\in\{1,3\}$. However, simple calculations show that these elements are all equal to $\un\tr,\uc\tr,\un\tc$ or $\uc\tc$. The following theorem establishes precisely how many distinct idempotent elements exist in the algebra of tricomplex numbers.\footnote{It is conjectured in \cite{Vallieres} that the algebra $\Mu{n},n\geq1$ contains exactly $2^{2^{n-1}}$ distinct idempotent elements.}
\begin{theorem}[See \cite{Vallieres}]\label{thm:16 idempotents}
There exists exactly sixteen distinct idempotent elements in $\TC$:
\begin{gather*}
0,\,1,\,\gamma_k = \frac{1+\jh{k}}{2},\,\overline{\gamma_k}=\frac{1-\jh{k}}{2},\quad\text{where }k=1,2,3,
\shortintertext{the four products}
\ut,\,\uct,\,\utc,\,\uctc
\intertext{and}
1-\ut,\,1-\uct,\,1-\utc\,\text{ and }\,1-\uctc.
\end{gather*}
\end{theorem}
The different types of conjugation of a tricomplex number have been studied in \cite{GarantPelletier,GarantRochon}. More precisely, it is shown that any tricomplex number $\eta$ has seven different conjugates, denoted $\eta^{\ddagger_{i}},\,i=1,\dotsc,7$ (see Figure \ref{eq:7conj}), and that their composition (together with the identity conjugation) forms a group isomorphic to $\mathbb{Z}_{2}^{3}$. Analogous results in $\Mu{n}$, $n\geq1$ are also presented. In addition, the authors of \cite{Holomorphy} provided a different but equivalent definition of conjugation valid in the general case.
%
%Ci-bas : display en tant que figure (with caption)
%
\begin{figure}[h]
\begin{equation*}
\begin{aligned}
\eta^{\ddagger_{1}} &:= z_1 + z_2\im2 - z_3\im3 - z_4\jh3\\
\eta^{\ddagger_{2}} &:= \overline{z_1} + \overline{z_2}\im2 + \overline{z_3}\im3 + \overline{z_4}\jh3\\
\eta^{\ddagger_{3}} &:= z_1 - z_2\im2 + z_3\im3 - z_4\jh3\\
\eta^{\ddagger_{4}} &:= \overline{z_1} - \overline{z_2}\im2 + \overline{z_3}\im3 - \overline{z_4}\jh3\\
\eta^{\ddagger_{5}} &:= \overline{z_1} + \overline{z_2}\im2 - \overline{z_3}\im3 - \overline{z_4}\jh3\\
\eta^{\ddagger_{6}} &:= z_1 - z_2\im2 - z_3\im3 + z_4\jh3\\
\eta^{\ddagger_{7}} &:= \overline{z_1} - \overline{z_2}\im2 - \overline{z_3}\im3 + \overline{z_4}\jh3
\end{aligned}
\end{equation*}
\caption{The seven tricomplex conjugates.}
\label{eq:7conj}
\end{figure}
%
%Ci-bas : Display en tant qu'équation (no caption, numéro 1.3 pour l'ensemble)
%
%\begin{equation}\label{eq:7conj}
%\begin{aligned}
%\eta^{\ddagger_{1}} &= z_1 + z_2\im2 - z_3\im3 - z_4\jh3\\
%\eta^{\ddagger_{2}} &= \overline{z_1} + \overline{z_2}\im2 + \overline{z_3}\im3 + \overline{z_4}\jh3\\
%\eta^{\ddagger_{3}} &= z_1 - z_2\im2 + z_3\im3 - z_4\jh3\\
%\eta^{\ddagger_{4}} &= \overline{z_1} - \overline{z_2}\im2 + \overline{z_3}\im3 - \overline{z_4}\jh3\\
%\eta^{\ddagger_{5}} &= \overline{z_1} + \overline{z_2}\im2 - \overline{z_3}\im3 - \overline{z_4}\jh3\\
%\eta^{\ddagger_{6}} &= z_1 - z_2\im2 - z_3\im3 + z_4\jh3\\
%\eta^{\ddagger_{7}} &= \overline{z_1} - \overline{z_2}\im2 - \overline{z_3}\im3 + \overline{z_4}\jh3
%\end{aligned}
%\end{equation}
%
%
%
%
%Ne pas oublier d'ajuster le mot «Figure» ou «Equation» dans le par. au-dessus selon l'apparence choisie
%
Nevertheless, one aspect of the theory that currently lacks understanding in the multicomplex setting $\Mu{n},n\geq3$ is the relationship between conjugation and invertibility. Indeed, although Price \cite{Baley} proposed various invertibility criteria based on the use of idempotent representations or Cauchy-Riemann matrices, none of these methods involve multicomplex conjugates, unlike the well known formula $z^{-1}=\frac{\overline{z}}{{|z|}^{2}}$ in the complex plane. The next new results, valid in the tricomplex case, represent a first step in this direction.\footnote{The reference \cite{Vallieres} contains two conjectures extending these results to $\Mu{n},n\geq1$.}
\begin{proposition}\label{prop:produit reel}
Let $\eta\in\TC$. The multiplication of $\eta$ by its seven different conjugates always equals a non-negative real number. In other words,
\[
\eta\eta^{\ddagger_{1}}\eta^{\ddagger_{2}}\eta^{\ddagger_{3}}\eta^{\ddagger_{4}}\eta^{\ddagger_{5}}\eta^{\ddagger_{6}}\eta^{\ddagger_{7}}=\eta\cdot\prod_{i=1}^{7}{\eta^{\ddagger_{i}}}\in\R_{\geq0}.
\]
\end{proposition}
\begin{proof}
A brute force approach is to perform straightforward but tedious calculations, which ultimately lead to the desired result. However, a more refined proof is available in \cite{Vallieres}.
\end{proof}
\begin{theorem}
Let $\eta\in\TC$. Then, $\eta$ is invertible if and only if
\[
\eta\cdot\prod_{i=1}^{7}{\eta^{\ddagger_{i}}}\neq 0.
\]
Moreover, when $\eta$ is invertible, its multiplicative inverse is given by the formula
\[
\eta^{-1}=\frac{\eta^{\ddagger_{1}}\eta^{\ddagger_{2}}\eta^{\ddagger_{3}}\eta^{\ddagger_{4}}\eta^{\ddagger_{5}}\eta^{\ddagger_{6}}\eta^{\ddagger_{7}}}{\eta\cdot\prod_{i=1}^{7}{\eta^{\ddagger_{i}}}}.
\]
\end{theorem}
\begin{proof}
This is a direct consequence of proposition \ref{prop:produit reel}.
\end{proof}
\noindent It is rather interesting to note that these results also uncover a link between conjugation and Cauchy-Riemann matrices, as the multiplication of any $\eta\in\TC$ by its seven conjugates equals the determinant of its Cauchy-Riemann matrix with real coefficients \cite{Vallieres}. 
Furthermore, we have the following corollary.
\begin{corollary}
Any non-zero tricomplex number is a zero divisor if and only if it cancels out with one of its conjugates or a product of these.
\end{corollary}

\subsection{Idempotent representations of a tricomplex number}\label{subsec:reps}
Let $n\geq2$ and consider the algebra $\Mu{n}$. Then, as in \cite{BrouilletteRochon,GarantRochon,Holomorphy,vajiac}, one can use the procedure developed by Price \cite{Baley} to construct $n-1$ sets of idempotent elements, each forming an orthogonal \textit{basis} of the space $\Mu{n}$, and find an idempotent representation in each of those bases. Thus, in the tricomplex case, two idempotent representations can be obtained this way. First, we consider the elements $\tr$ and $\tc$, to which corresponds the identity \cite{Baley}
\begin{equation}\label{eq:idemp tr}
\eta = (\eta_1-\eta_2\im2)\tr + (\eta_1+\eta_2\im2)\tc
\end{equation}
that we will call the \tr-representation of $\eta=\eta_1+\eta_2\im3\in\TC$. The second representation can be derived from the first by noticing that the two idempotent components in \eqref{eq:idemp tr} are bicomplex numbers. Hence, writing $\eta_1=\eta_{11}+\eta_{12}\im2,\,\eta_2=\eta_{21}+\eta_{22}\im2$ and using \eqref{eq:idemp BC} yields
\begin{equation}\label{eq:idemp 4}
\eta = \eta_{\ut}\cdot\ut + \eta_{\uct}\cdot\uct + \eta_{\utc}\cdot\utc + \eta_{\uctc}\cdot\uctc
\end{equation}
where
\begin{align*}
\eta_{\ut} &= (\eta_{11}+\eta_{22}) - (\eta_{12}-\eta_{21})\im1;\\
\eta_{\uct} &= (\eta_{11}+\eta_{22}) + (\eta_{12}-\eta_{21})\im1;\\
\eta_{\utc} &= (\eta_{11}-\eta_{22}) - (\eta_{12}+\eta_{21})\im1;\\
\eta_{\uctc} &= (\eta_{11}-\eta_{22}) + (\eta_{12}+\eta_{21})\im1.
\end{align*}
Remark that every idempotent component above is a complex number: thus, this process cannot be repeated. However, there is no pretense of completeness regarding the $n-1$ sets of idempotent elements obtained by applying Price's method. For example, in the algebra of tricomplex numbers, \autoref{thm:16 idempotents} states that there exists fourteen non-trivial idempotent elements, whereas only six of them are used in identities \eqref{eq:idemp tr} and \eqref{eq:idemp 4}. This suggests that additional representations involving other idempotent elements exist in $\TC$. Indeed, by solving the system of linear equations associated to the equality $\eta=a\dx + b\dc$ where $a,b\in\BC$, one obtains the solution $\{a=\eta_1-\eta_2\im1,b=\eta_1+\eta_2\im1\}$, which leads to the identity
\begin{equation}\label{eq:idemp de}
\eta = (\eta_1-\eta_2\im1)\dx + (\eta_1+\eta_2\im1)\dc.
\end{equation}

Note that the properties of the idempotents elements involved in identities \eqref{eq:idemp tr} to \eqref{eq:idemp de} ensure that the algebraic operations in $\TC$ can be carried out term by term in these representations. More generally, this is also true for the idempotent representations in $\Mu{n}$ for $n\geq2$, and in fact, this property is the cornerstone of several important results related to generalized Mandelbrot and filled-in Julia sets \cite{ClaudiaRochon, BrouilletteRochon,GarantRochon}. In a similar fashion, representations \eqref{eq:idemp BC} to \eqref{eq:idemp de} will, most of the time, be the key to establish the results presented in sections \ref{sec:geometrical} and \ref{sec:cube}.

\section{Geometrical characterizations}\label{sec:geometrical}
\subsection{The tricomplex Mandelbrot set $\M3$}
We begin this section by recalling relevant definitions introduced in \cite{BrouilletteRochon,GarantRochon}. Let $P_{c}(\eta)=\eta^2+c$ and denote
\[
P_{c}^{(n)}(\eta)=\underbrace{(P_{c}\circ P_{c}\circ\cdots\circ P_{c})}_{n\text{ times}}(\eta).
\]
It follows that the classical Mandelbrot set can be defined as
\[
\M{1} = \{c\in\Mu{1} : \suite\text{ is bounded}\}.
\]
If we set $c\in\Mu{2}$ or $c\in\Mu{3}$ instead, we have the following generalizations.
\begin{definition}
The bicomplex Mandelbrot set is defined as
\[
\M{2}=\{c\in\BC:\suite\text{ is bounded}\}.
\]
\end{definition}
\begin{definition}
The tricomplex Mandelbrot set is defined as
\[
\M{3}=\{c\in\TC:\suite\text{ is bounded}\}.
\]
\end{definition}
The bicomplex Mandelbrot set was first studied by Rochon in \cite{Rochon1} while its tricomplex analog was introduced in \cite{GarantRochon}. In addition, these generalizations are particular cases of the multicomplex Multibrot sets studied in \cite{BrouilletteRochon}. For the sake of clarity and continuity, we will use the same subspaces and notation as those in the last two references. Now, since $\M3$ is an eight-dimensional object, we are only able to visualize its various projections on tridimensional subspaces of \TC, called 3D slices. This brings us to the following definitions.
\begin{definition}
Let $\im{k},\im{l},\im{m}\in\{1,\im1,\im2,\im3,\im4,\jh1,\jh2,\jh3\}$ with $\im{k}\neq\im{l},\,\im{k}\neq\im{m}$ and $\im{l}\neq\im{m}$. The space
\[
\mT(\im{k},\im{l},\im{m}) := \vspan\{\im{k},\im{l},\im{m}\}
\]
is the vector subspace of $\TC$ consisting of all real finite linear combinations of these three distinct units.
\end{definition}
\begin{definition}
Let $\im{k},\im{l},\im{m}\in\{1,\im1,\im2,\im3,\im4,\jh1,\jh2,\jh3\}$ with $\im{k}\neq\im{l},\,\im{k}\neq\im{m}$ and $\im{l}\neq\im{m}$. We define a principal 3D slice of the tricomplex Mandelbrot set $\M3$ as
\begin{align*}
\T(\im{k},\im{l},\im{m}) &= \{c\in\mT(\im{k},\im{l},\im{m}) : \suite\text{ is bounded}\} \\
&= \mT(\im{k},\im{l},\im{m})\cap\M{3}.
\end{align*}
\end{definition}
As noted earlier, the fifty-six 3D slices corresponding to the various combinations of $\im{k},\im{l},\im{m}\in\{1,\im1,\im2,\im3,\im4,\jh1,\jh2,\jh3\}$ have been classified according to their dynamics (and, consequently, their appearance in visualization software) in \cite{GarantRochon}. \autoref{fig:8 slices} illustrates the eight principal 3D slices of the tricomplex Mandelbrot set (power 2) resulting from this characterization \cite{BrouilletteRochon}.
\begin{figure}[ht]%% J'ai mis 8 subfigures plutôt qu'une image regroupant les 8 coupes afin de pouvoir référer à chaque coupe de façon ciblée (label).
\centering
	\begin{subfigure}{0.243\textwidth}
	\centering
	\includegraphics[width=2.9cm]{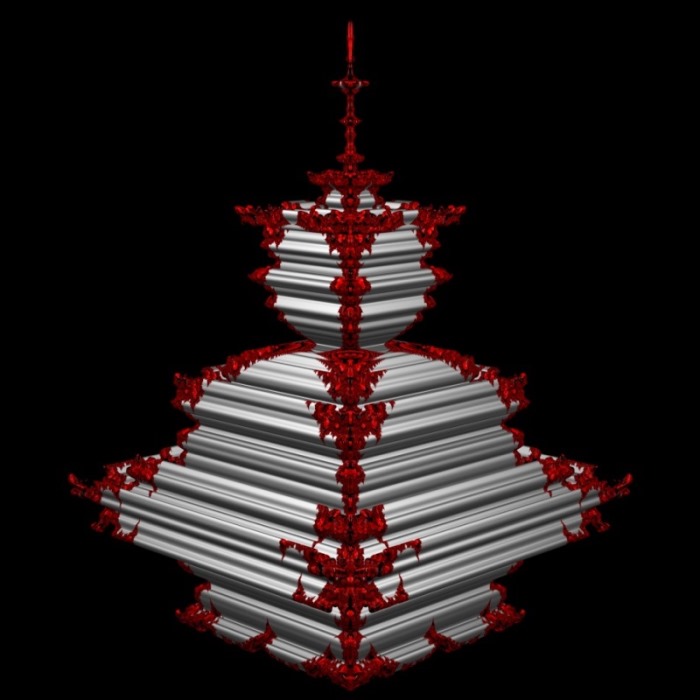}
	\caption{$\T(1,\im1,\im2)$}
	\label{fig:tetrabrot}
	\end{subfigure}
\hfill
	\begin{subfigure}{0.243\textwidth}
	\centering
	\includegraphics[width=2.9cm]{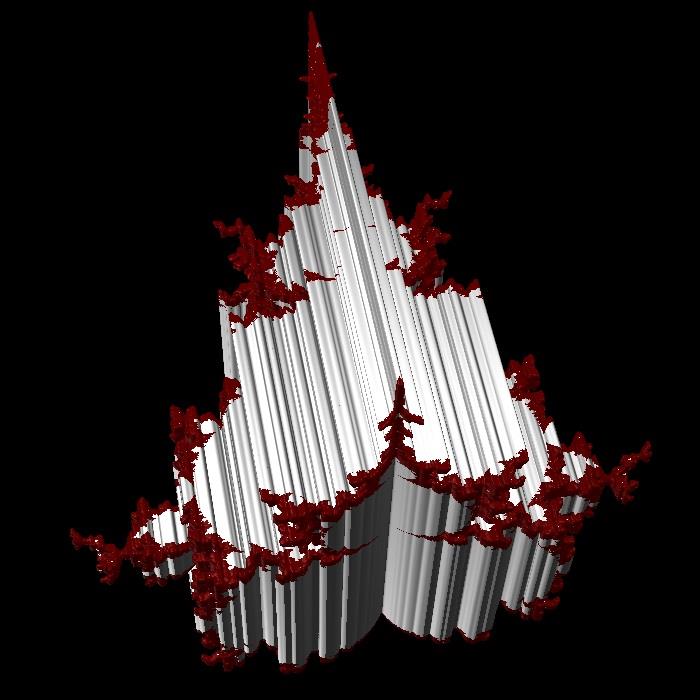}
	\caption{$\T(1,\im1,\jh1)$}
	\label{fig:arrowheadbrot}
	\end{subfigure}
\hfill
	\begin{subfigure}{0.243\textwidth}
	\centering
	\includegraphics[width=2.9cm]{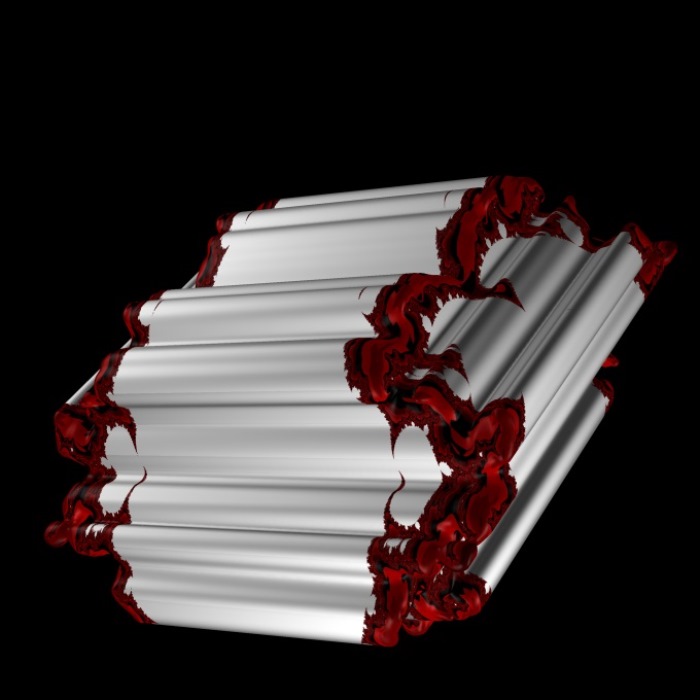}
	\caption{$\T(\im1,\im2,\jh1)$}
	\label{fig:mousebrot}
	\end{subfigure}
\hfill
	\begin{subfigure}{0.243\textwidth}
	\centering
	\includegraphics[width=2.9cm]{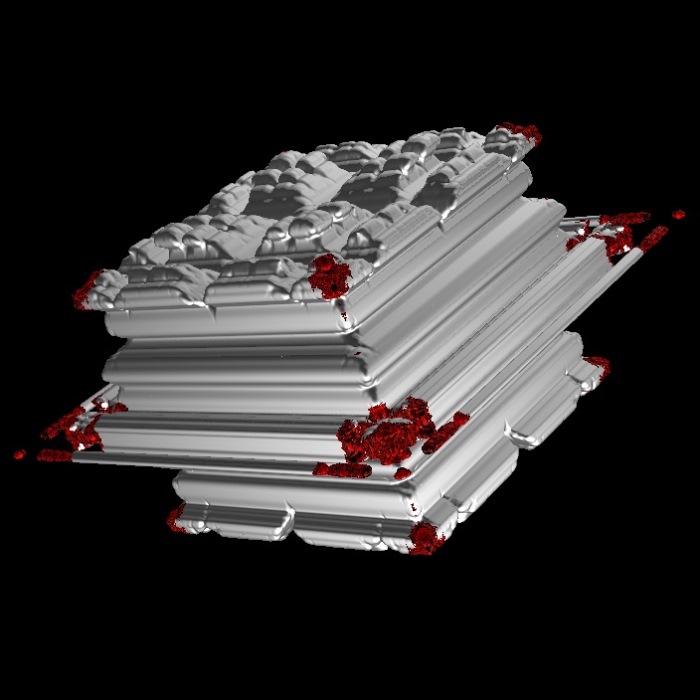}
	\caption{$\T(\im1,\im2,\jh2)$}
	\label{fig:turtlebrot}
	\end{subfigure}
\\ \vspace*{5pt}
	\begin{subfigure}{0.243\textwidth}
	\centering
	\includegraphics[width=2.9cm]{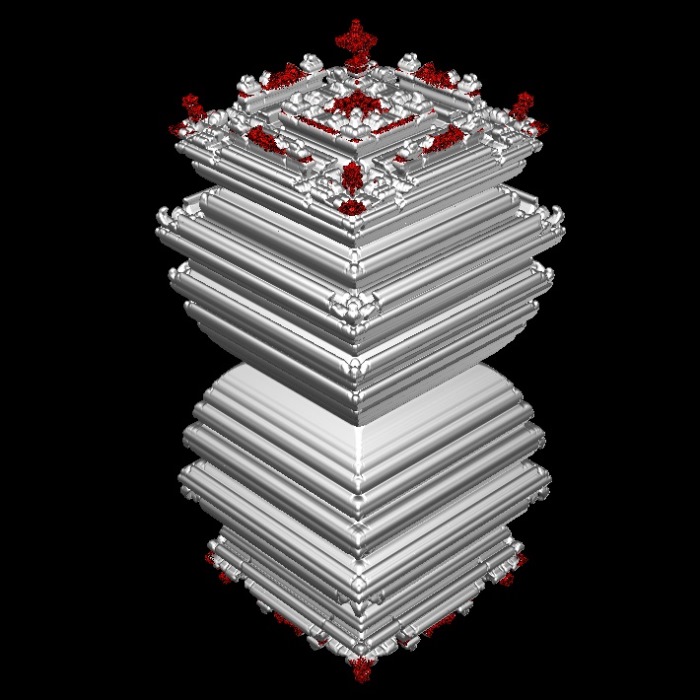}
	\caption{$\T(\im1,\jh1,\jh2)$}
	\label{fig:hourglassbrot}
	\end{subfigure}
\hfill
	\begin{subfigure}{0.243\textwidth}
	\centering
	\includegraphics[width=2.9cm]{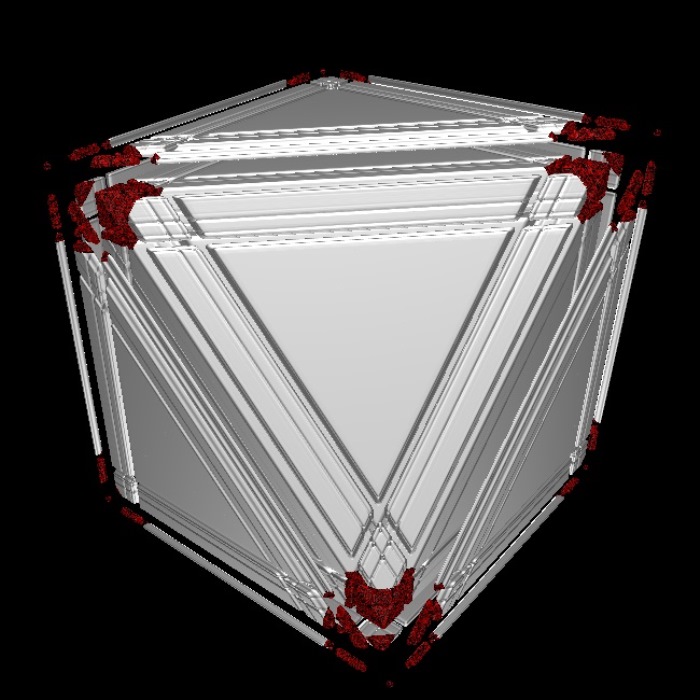}
	\caption{$\T(\im1,\im2,\im3)$}
	\label{fig:metabrot}
	\end{subfigure}
\hfill
	\begin{subfigure}{0.243\textwidth}
	\centering
	\includegraphics[width=2.9cm]{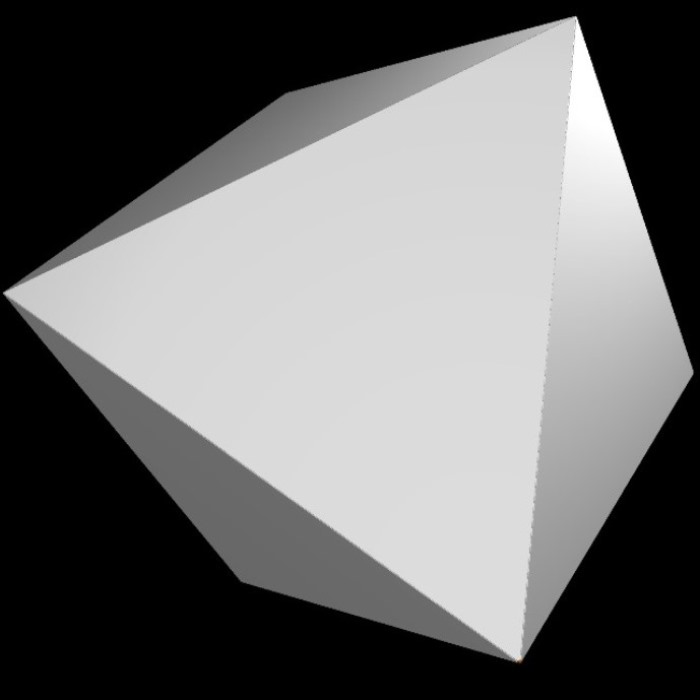}
	\caption{$\T(1,\jh1,\jh2)$}
	\label{fig:airbrot}
	\end{subfigure}
\hfill
	\begin{subfigure}{0.243\textwidth}
	\centering
	\includegraphics[width=2.9cm]{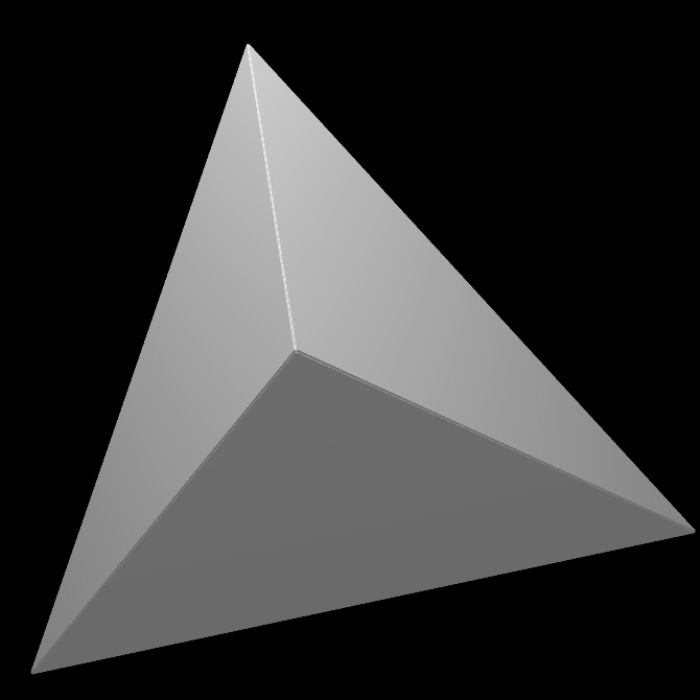}
	\caption{$\T(\jh1,\jh2,\jh3)$}
	\label{fig:firebrot}
	\end{subfigure}
\caption{The eight principal 3D slices of $\M3$.}
\label{fig:8 slices}
\end{figure}
Among these are two principal 3D slices called the Tetrabrot (Fig.\ref{fig:tetrabrot}) and the Airbrot (Fig.\ref{fig:airbrot}), for which geometrical characterizations have been respectively developed in \cite{Rochon1} and \cite{GarantRochon}. For the latter, doing so indirectly proved that the Airbrot is a regular octahedron, thus confirming the presence of a first Platonic solid within tricomplex dynamics.

\subsection{Characterizations of the principal 3D slices}
Essentially, we want to elaborate similar characterizations for the remaining principal slices in order to confirm or explain some of their properties. One notable example is the fact that the Firebrot (Fig.\ref{fig:firebrot}) is conjectured to be a regular tetrahedron \cite{GarantPelletier}. This assertion will be the subject of \autoref{thm:tetraedre}. Unless otherwise stated, the following results are an excerpt of those established in \cite{Vallieres}.
\begin{proposition}[See \cite{Rochon1}]\label{Tetrabrot}
The Tetrabrot can be characterized as follows:
\[
\mathcal{T}(1,\im1,\im2) = \underset{y\in [-m,m]}{\bigcup}\{[(\M1-y\im1)\cap(\M1+y\im1)] + y\im2\}
\]
where
\[
m:=\sup\{q\in\mR:\exists p\in\mR\text{ such that }p+q\im1\in\M1\}.
\]
\end{proposition}
\begin{proposition}\label{Arrowhead}
The principal 3D slice $\T(1,\im1,\jh1)$, named the Arrowheadbrot (Fig.\ref{fig:arrowheadbrot}), can be expressed as follows:
\[
\mathcal{T}(1,\im1,\jh1) = \underset{y\in \left[-\frac{9}{8},\frac{9}{8}\right]}{\bigcup}\{[(\M1-y)\cap(\M1+y)] + y\jh1\}.
\]
\end{proposition}
\begin{proof}
Let $c\in\T(1,\im1,\jh1)$. Then, $c=c_{1} + c_{2}\im1 + c_{4}\jh1$ and by using identity \eqref{eq:idemp BC}, we can write
\begin{align*}
c &= c_{1} + c_{2}\im1 + (c_{4}\im1)\im2 \\
&= (c_{1} + c_{2}\im1 +  c_{4})\un + (c_{1} + c_{2}\im1 - c_{4})\uc \\
&= (d + c_{4})\un + (d - c_{4})\uc,
\end{align*}
where $d=c_{1} + c_{2}\im1 \in\Co(\im1)$. By remarking that $\suite$ is bounded if and only if $\{P_{d+c_4}^{n}(0)\}_{n\in\N}$ and $\{P_{d-c_4}^{n}(0)\}_{n\in\N}$ are bounded, we deduce that $c\in\T(1,\im1,\jh1)$ if and only if $d\pm c_4\in\M1$. Then, since
\[
\M{1}-z=\{c\in\Co(\im1):\{P_{c+z}^{n}(0)\}_{n\in\N}\text{ is bounded}\}\,\forall z\in\Co(\im1),
\]
we have $d\pm c_4\in\M1$ if and only if $d\in(\M{1}-c_{4})\cap(\M{1}+c_{4})$. Therefore,
\begin{align*}
\T(1,\im1,\jh1) &= \{c\in\Tbb(1,\im1,\jh1):\suite\text{ is bounded}\} \\
&= \{c_{1} + c_{2}\im1 + c_{4}\jh1:c_{1} + c_{2}\im1\in(\M{1}-c_{4})\cap(\M{1}+c_{4})\}\\
&=\underset{y\in\R}{\bigcup}\{[(\M{1}-y)\cap(\M{1}+y)] + y\jh1\}.
\end{align*}
It is possible to be more precise with the last expression by using the well known property $\M{1}\cap\R = \left[-2,\frac{1}{4}\right]$ (see \cite{Carleson, RochonRansfordParise}). Indeed, we infer from it that $(\M1-y)\cap(\M1+y)=\emptyset$ whenever $y\in{\left[-\frac{9}{8},\frac{9}{8}\right]}^{c}$. Thus,
\[
\T(1,\im1,\jh1) = \underset{y\in \left[-\frac{9}{8},\frac{9}{8}\right]}{\bigcup}\{[(\M1-y)\cap(\M1+y)] + y\jh1\},
\]
and since $\M{1}\cap\R$ is an interval, $(\M1-y)\cap(\M1+y)\neq\emptyset,\forall y\in\left[-\frac{9}{8},\frac{9}{8}\right]$.
\end{proof}
The above results show that both the Tetrabrot and the Arrowheadbrot can be described as the union of intersections of two translated classical Mandelbrot sets (along the imaginary axis for the former, and along the real axis for the latter). In other words, these tridimensional fractals can be obtained by using multiple copies of a two-dimensional related fractal. It is rather interesting to note that this is the case for each of the eight principal 3D slices of $\M3$, although the required 2D subsets vary. In fact, for a given principal 3D slice, the properties of a related subspace called its “iterates space” determine the type and number of 2D subsets involved in its geometrical characterizations.\footnote{The \textit{iterates space} of a principal slice was first introduced in \cite{Brouillette,BrouilletteRochon} for other purposes.} A detailed analysis on the subject is available in \cite{Vallieres}.

\begin{proposition}\label{Metabrot}
The principal slice $\T(\im1,\im2,\im3)$, called the Metabrot (Fig.\ref{fig:metabrot}), can be characterized as follows:
\[
\T(\im1,\im2,\im3) = \underset{y\in[-m,m]}{\bigcup}\{[(A-y\im2)\cap(A+y\im2)] + y\im3\}
\]
where 
\[
A:=\{a\in\vspan\{\im1,\im2\}:\{P_{a}^{(n)}(0)\}_{n\in\N}\text{ is bounded }\forall n\in\N\}
\]
and
\[
m:=\sup\{q\in\mR:\exists p\in\mR\text{ such that }p+q\im1\in\M1\}.
\]
\end{proposition}
\begin{proof}
Let $c\in\T(\im1,\im2,\im3)\subset\Tbb(\im1,\im2,\im3)$. Using identity \eqref{eq:idemp tr} yields
\begin{align*}
c &= (c_{2}\im1 + c_{3}\im2) + (c_{5})\im3 \\
&= ((c_{2}\im1 + c_{3}\im2) - c_{5}\im2)\tr + ((c_{2}\im1 + c_{3}\im2) + c_{5}\im2)\tc\\
&= (d - c_{5}\im2)\tr + (d + c_{5}\im2)\tc,
\end{align*}
where $d=c_{2}\im1 + c_{3}\im2$. It is not too difficult to see that the sequence $\suite$ is bounded if and only if the sequences $\{P_{d-c_{5}\im2}^{(n)}(0)\}_{n\in\N}$ and $\{P_{d+c_{5}\im2}^{(n)}(0)\}_{n\in\N}$ are bounded. Then, consider the set
\begin{align}\label{eq: ensemble A}
A:\!&=\{a\in\vspan\{\im1,\im2\}:\{P_{a}^{(n)}(0)\}_{n\in\N}\text{ is bounded }\forall n\in\N\} \notag\\
&= \M{2}\cap\Tbb(\im1,\im2).
\end{align}
It follows that $\{P_{d-c_{5}\im2}^{(n)}(0)\}_{n\in\N}$ and $\{P_{d+c_{5}\im2}^{(n)}(0)\}_{n\in\N}$ are bounded if and only if $d\mp c_{5}\im2\in A$, and that is the case if and only if $d\in(A-c_{5}\im2)\cap(A+c_{5}\im2)$. Consequently,
\begin{align*}
\T(\im1,\im2,\im3) &= \{c_{2}\im1 + c_{3}\im2 + c_5\im3:c_{2}\im1 + c_{3}\im2\in(A-c_{5}\im2)\cap(A+c_{5}\im2)\}\\
&= \underset{y\in\R}{\bigcup}\{[(A-y\im2)\cap(A+y\im2)] + y\im3\}.
\end{align*}
In order to be more precise with the last expression, we need to remark that $A\subset\{c_{2}\im1 + c_{3}\im2:|c_3|\leq m\}$. Indeed, let $a=a_2\im1+a_3\im2\in A$. Then, identity \eqref{eq:idemp BC} implies we can write $a=(a_2-a_3)\im1\un + (a_2+a_3)\im1\uc$. In addition, the Mandelbrot set $\M1$ is symmetric along the real axis, meaning that $\forall x\in\R, x\im1\in\M{1} \Leftrightarrow -x\im1\in\M{1}$. It follows that
\begin{align}\label{eq: carac A}
a\in A &\Leftrightarrow \{P_{a}^{n}(0)\}_{n\in\N}\text{ is bounded} \notag\\
&\Leftrightarrow (a_2 \pm a_3)\im1 \in \M{1} \notag\\
&\Leftrightarrow (\pm a_2 \pm a_3)\im1 \in \M{1},
\end{align}
hence the aforementioned property. We conclude that $(A+y\im2)\cap(A-y\im2)=\emptyset$ whenever $y\in[-m,m]^c$, whence
\[
\T(\im1,\im2,\im3) = \underset{y\in[-m,m]}{\bigcup}\{[(A-y\im2)\cap(A+y\im2)] + y\im3\}. \qedhere
\]
\end{proof}
\begin{figure}[h]
\centering
\includegraphics[width=7cm]{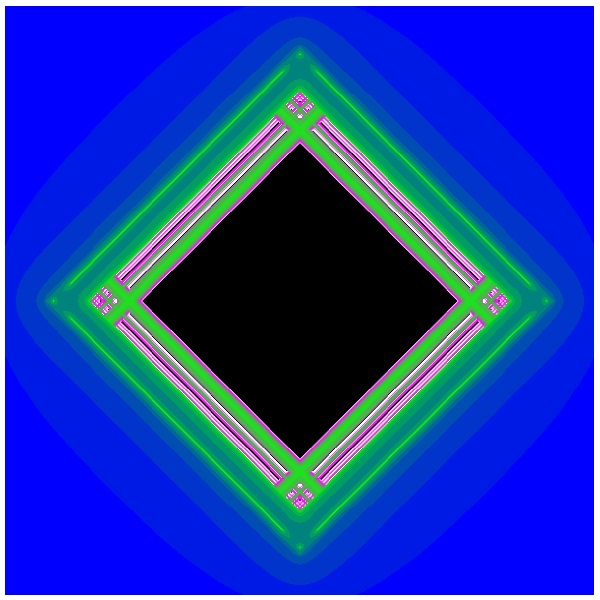}
\caption{The set $A\subset\vspan\{\im1,\im2\}$.}
\label{fig:Ensemble A}
\end{figure}
%-------------------------------------------------- Version analytique (4 composantes) --------------------------------
%
%\begin{proposition}\label{Metabrot}
%The principal slice $\T(\im1,\im2,\im3)$, called the Metabrot (Fig.\ref{fig:metabrot}), can be characterized as follows:
%\[
%\T(\im1,\im2,\im3) = \{c_{2}\im1 + c_{3}\im2 + c_{5}\im3 : (\pm c_2\pm c_3 \pm c_5)\im1 \in\M{1}\}.
%\]
%\end{proposition}
%%
%\begin{proof}
%Let $c\in\Tbb(\im1,\im2,\im3)$ such that $\suite$ is bounded. Using the identity \eqref{eq:idemp 4} yields 
%\begin{align*}
%c &= (c_{2}\im1) + (c_{3})\im2 + (c_{5})\im3\\[5pt]
%&= \begin{multlined}[t]
%	(c_{2}\im1-(c_3-c_5)\im1)\ut + (c_{2}\im1+(c_3-c_5)\im1)\uct \\
%	+ (c_{2}\im1-(c_3+c_5)\im1)\utc + (c_{2}\im1+(c_3+c_5)\im1)\uctc 
%	\end{multlined}\\[5pt]
%&= \begin{multlined}[t]
%	(c_2-c_3+c_5)\im1\ut + (c_2+c_3-c_5)\im1\uct \\
%	+ (c_2-c_3-c_5)\im1\utc + (c_2+c_3+c_5)\im1\uctc.
%	\end{multlined}
%\end{align*}
%It is not too difficult to see that the sequence $\suite$ is bounded if and only if $(c_2\pm c_3 \pm c_5)\im1 \in \M{1}$. In addition, the Mandelbrot set $\M1$ is symmetric along the real axis, meaning that $\forall x\in\R, x\im1\in\M{1} \Leftrightarrow -x\im1\in\M{1}$. Therefore,
%\begin{align*}
%c \in \T(\im1,\im2,\im3) &\Leftrightarrow \suite\text{ is bounded}\\
%&\Leftrightarrow (c_2\pm c_3 \pm c_5)\im1 \in \M{1}\\
%&\Leftrightarrow (\pm c_2\pm c_3 \pm c_5)\im1 \in \M{1},
%\end{align*}
%hence the result.
%\end{proof}
%-------------------------------------------------------------------------------------------------------------------------------
%
%
%
\autoref{fig:Ensemble A} illustrates that the set $A$ looks like a filled-in square with a fractal perimeter. Although this might seem intriguing, equivalence \eqref{eq: carac A} provides a simple explanation for this phenomenon. Indeed, since $a_2,a_3\in\R$, the visual appearance of the set $A$ is entirely determined by the peculiar dynamics of the classical Mandelbrot set $\M1$ along the imaginary axis.

Furthermore, as equation \eqref{eq: ensemble A} points out, we can view the set $A$ as a projection of the bicomplex Mandelbrot set $\M2$ onto the vector subspace $\Tbb(\im1,\im2)$, just like the classical Mandelbrot set can be viewed as a projection of $\M2$ onto the complex plane $\Tbb(1,\im1)$. 
It is then natural to ask whether the projection of $\M2$ onto the subspace $\Tbb(1,\jh1)$, which is the vector space of hyperbolic numbers, plays a significant role within tricomplex dynamics.

Let us start by recalling the relevant definition. The set of hyperbolic numbers (see \cite{Sobczyk, vajiac2}) is defined as
\[
\D:=\{x+y\jh1:x,y\in\R\text{ and }\jh1^2=1\}.
\]
From this, it is evident that $\D \subset \BC$. In 1990, Senn \cite{Senn} was the first to apply the Mandelbrot algorithm to hyperbolic numbers in order to generate another 2D set that revealed itself as a simple filled-in square. This property was later proved by Metzler \cite{Metzler}. The hyperbolic Mandelbrot set (or Hyperbrot) is defined as follows:
\[
\Hy := \{c\in\D : \suite\text{ is bounded}\}
\]
while Metzler obtained the following characterization:
\[
\Hy = \left\{(a,b)\in{\R}^{2}:\left|a+\frac{7}{8}\right|+|b|\leq\frac{9}{8}\right\}.
\]
In terms of tricomplex dynamics, the Hyperbrot is the key to generate Platonic solids in specific three-dimensional subspaces of \TC. Indeed, it can be shown that two principal 3D slices of $\M3$ can be characterized as such, and that is the content of the next two propositions.
\begin{proposition}[See \cite{GarantRochon}]\label{Airbrot}
The Airbrot admits the following representation:
\[
\T(1,\jh1,\jh2)= \underset{y\in [-\frac{9}{8},\frac{9}{8}]}{\bigcup}\{[(\Hy-y\jh1)\cap(\Hy+y\jh1)] + y\jh2\}.
\]
\end{proposition}
As stated in \cite{GarantRochon}, proposition \ref{Airbrot} and Metzler's characterization establish that the Airbrot is a regular octahedron of edge length equal to $\frac{9}{8}\sqrt{2}$. We wish to prove a similar result for the Firebrot, whose geometrical shape also seems regular.
\begin{proposition}\label{Firebrot characterization}
The principal slice $\T(\jh1,\jh2,\jh3)$, named the Firebrot (Fig.\ref{fig:firebrot}), can be characterized as follows:
\[
\T(\jh1,\jh2,\jh3) = \underset{y \in [-\frac{1}{4},\frac{1}{4}]}{\bigcup}\{[(\Hp+y\jh1)\cap(-\Hp-y\jh1)] + y\jh2\}
\]
where $\Hp := \{c_{7}\jh3+c_{4}\jh1:c_{7}+c_{4}\jh1 \in \Hy\}$.
\end{proposition}
\begin{proof}
For any $c\in\T(\jh1,\jh2,\jh3)\subset\Tbb(\jh1,\jh2,\jh3)$, we have
\begin{align*}
c &= c_{4}\jh1 + c_{6}\jh2 + c_{7}\jh3\\
&= (d-c_{6}\jh1)\tr + (-\overline{d}+c_{6}\jh1)\tc,
\end{align*}
where $d=c_{7} + c_{4}\jh1\in\D$ and $-\overline{d}=-c_{7} + c_{4}\jh1$.\footnote{The hyperbolic conjugate of any $z=x+y\jh1\in\D$ is defined as $\overline{z}=x-y\jh1$.} Again, since $\suite$ is bounded if and only if $\{P_{d-c_{6}\jh1}^{(n)}(0)\}_{n\in\N}$ and $\{P_{-\overline{d}+c_{6}\jh1}^{(n)}(0)\}_{n\in\N}$ are bounded, we have $c\in\T(\jh1,\jh2,\jh3)$ if and only if $d\in(\Hy+c_{6}\jh1)\cap(-\Hy-c_{6}\jh1)$. Therefore,
\[
\T(\jh1,\jh2,\jh3) = \{c_{4}\jh1 + c_{6}\jh2 + c_{7}\jh3:c_{7}+c_{4}\jh1\in(\Hy+c_{6}\jh1)\cap(-\Hy-c_{6}\jh1)\}.
\]
Notice that $\Hy\not\subset\Tbb(\jh1,\jh2,\jh3)\!\supset\!\T(\jh1,\jh2,\jh3)$. In order to establish a characterization analogous to that of the previous 3D slices, we must make sure that the 2D sets involved are in the right subspace. This can be achieved by setting
\[
\Hp := \{c_{7}\jh3+c_{4}\jh1:c_{7}+c_{4}\jh1 \in \Hy\},
\]
where $\Hp$ is a duplicate of the Hyperbrot in the subspace $\Tbb(\jh1,\jh3)$. This way, we may write
\begin{align*}
\T(\jh1,\jh2,\jh3) &= \{c_{4}\jh1 + c_{6}\jh2 + c_{7}\jh3:c_{7}\jh3+c_{4}\jh1\in(\Hp+c_{6}\jh1)\cap(-\Hp-c_{6}\jh1)\}\\
&= \underset{y\in\R}{\bigcup}\{[(\Hp+y\jh1)\cap(-\Hp-y\jh1)] + y\jh2\}.
\end{align*}
To complete the proof, remark that $(\Hp+y\jh1)\cap(-\Hp-y\jh1)=\emptyset,\,\forall y\in{\left[-\frac{1}{4},\frac{1}{4}\right]}^{c}$. Therefore,
\[
\T(\jh1,\jh2,\jh3) = \underset{y \in \left[-\frac{1}{4},\frac{1}{4}\right]}{\bigcup}\{[(\Hp+y\jh1)\cap(-\Hp-y\jh1)] + y\jh2\}.
\]
Finally, we could verify that $\forall y\in\left[-\frac{1}{4},\frac{1}{4}\right]$, the sets $(\Hp+y\jh1)\cap(-\Hp-y\jh1)$ are non-empty rectangles, hence the result.
\end{proof}
Propositions \ref{Airbrot} and \ref{Firebrot characterization} are interesting and similar in that they state how both the Airbrot and the Firebrot are Platonic solids that can be generated by the union of intersections of translated squares (Hyperbrots). The former can be generated using vertically translated Hyperbrots (that is, a translation along the hyperbolic axis), which ensures each intersection is square shaped. For the latter, a reflection along the hyperbolic axis is applied to one of the Hyperbrots prior to the vertical translation, meaning that every non-empty intersection will be rectangular. This particularity is what accounts for the Firebrot's tetrahedral shape.
\begin{theorem}\label{thm:tetraedre}
$\T(\jh1,\jh2,\jh3)$ is a regular tetrahedron with edge length of $\frac{\sqrt{2}}{2}$.
\end{theorem}
Although proposition \ref{Firebrot characterization} provides relevant insight into the Firebrot's geometrical shape, the assertion above is proven using a more direct approach.
\begin{proof}
Suppose that $c\in\T(\jh1,\jh2,\jh3)$ with $c=c_{4}\jh1+c_{6}\jh2+c_{7}\jh3$. Using identity \eqref{eq:idemp 4} yields
\begingroup
\addtolength{\jot}{0.5em}
\begin{align*}
c &= (0c_1 + 0c_2\im1) + (0c_3 + c_4\im1)\im2 + [(0c_5+c_6\im1) + (c_7 + 0c_8\im1)\im2]\im3 \\
&= \begin{multlined}[t]
(c_{7} - (c_{4}\im1-c_{6}\im1)\im1)\ut + (c_{7} + (c_{4}\im1-c_{6}\im1)\im1)\uct \\
+ (-c_{7} - (c_{4}\im1+c_{6}\im1)\im1)\utc +(-c_{7} + (c_{4}\im1+c_{6}\im1)\im1)\uctc
\end{multlined} \\
&= \begin{multlined}[t]
(c_{4}-c_{6}+c_{7})\ut + (-c_{4} + c_{6}+c_{7})\uct \\
+ (c_{4} + c_{6}-c_{7})\utc +(-c_{4} - c_{6}-c_{7})\uctc
\end{multlined} \\
&= (a_1)\ut + (a_2)\uct + (a_3)\utc +(a_4)\uctc,
\end{align*}
\endgroup
where we denote every component of the last equality by $a_i$ for convenience. It follows that $\suite$ is bounded if and only if $\{P_{a_{i}}^{(n)}(0)\}_{n\in\mathbb{N}}$ is bounded, $i=1,\dotsc,4$. 
Since $a_{i}\in\R$ $\forall i\in\{1,\dotsc,4\}$ and $\M{1}\cap\R=\left[-2,\frac{1}{4}\right]$, we deduce that
\begin{align*}
a_{1},a_{2},a_{3},a_{4}\in\M{1} \, %&\Leftrightarrow\, \left\{
%					\begin{array}{l}
%					-2 \leq c_{4}-c_{6}+c_{7} \leq \frac{1}{4} \\ [3pt]
%					-2 \leq -c_{4} + c_{6}+c_{7} \leq \frac{1}{4} \\ [3pt]
%					-2 \leq c_{4} + c_{6}-c_{7} \leq \frac{1}{4} \\ [3pt]         %Partie masquée. À rajouter au besoin
%					-2 \leq -c_{4} - c_{6}-c_{7} \leq \frac{1}{4}
%					\end{array}
%				\right. \\
&\Leftrightarrow\, \left\{
					\begin{array}{l}
						c_{4} - c_{6} + c_{7} \leq \frac{1}{4} \\ [1pt]
						-c_{4}+ c_{6} - c_{7} \leq 2 \\[1pt]
						-c_{4} + c_{6}+ c_{7} \leq \frac{1}{4} \\[1pt]
						c_{4} - c_{6} - c_{7} \leq 2\\[1pt]
						c_{4} + c_{6}-c_{7} \leq \frac{1}{4} \\[1pt]
						-c_{4} - c_{6} + c_{7} \leq 2 \\[1pt]
						-c_{4} - c_{6}-c_{7} \leq \frac{1}{4} \\[1pt]
						c_{4} + c_{6} + c_{7} \leq 2.
					\end{array}%
				\right.%
\end{align*}%
By using Fourier-Motzkin elimination (see \cite{Vallieres}, annex A), it is possible to show that $c\in\T(\jh1,\jh2,\jh3)\Rightarrow|c_6|\leq\frac{1}{4}$. In turn, this means that the second, fourth, sixth and eighth inequalities of the last system are unnecessary. In other words, they are redundant constraints and they can be removed without changing the set of solutions. Doing so reduces the number of inequalities to four:
\begin{equation*}
\begin{array}{l}
c_{4}-c_{6}+c_{7} \leq \frac{1}{4} \\ %[3pt]
-c_{4} + c_{6}+c_{7} \leq \frac{1}{4} \\ %[3pt]
c_{4} + c_{6}-c_{7} \leq \frac{1}{4} \\ %[3pt]
-c_{4} - c_{6}-c_{7} \leq \frac{1}{4}
\end{array}
\end{equation*}
The rest of the proof is carried out by determining the $(\jh1,\jh2,\jh3)$ coordinates of the four extreme points (vertices) of this system, and then calculating the distance between these.
\end{proof}

We now turn our attention to the remaining principal slices, which are called the Mousebrot (Fig.\ref{fig:mousebrot}), the Turtlebrot (Fig.\ref{fig:turtlebrot}) and the Hourglassbrot (Fig.\ref{fig:hourglassbrot}). Although the three of them admit geometrical characterizations similar to those presented at the beginning of the current section, the emphasis is placed on their interrelationships.
\begin{proposition}\label{prop:turtlebrot}
Consider the principal 3D slice $\T(\im1,\im2,\jh2)$. We have
\begin{gather*}
\T(\im1,\im2,\jh2) = \Ta(\im1,\im2,-\jh1)\mcap\Ta(\im1,\im2,\jh1)
\shortintertext{where}
\Ta(\im1,\im2,\jh1) := \{c_2\im1+c_3\im2+c_6\jh2:c_2\im1+c_3\im2+c_6\jh1 \in \T(\im1,\im2,\jh1)\}.
\end{gather*}
\end{proposition}
\begin{proof}
Suppose that $c\in\T(\im1,\im2,\jh2)\subset\Tbb(\im1,\im2,\jh2)$. Applying identity \eqref{eq:idemp tr} to $c$ yields
\begin{align*}
c &= c_2\im1+c_3\im2+(c_6\im1)\im3 \\
&= (c_2\im1+c_3\im2-c_6\jh1)\tr + (c_2\im1+c_3\im2+c_6\jh1)\tc \\
&= (-d^{\ddagger_4})\tr + (d)\tc,
\end{align*}
where $d=c_2\im1+c_3\im2+c_6\jh1 \Rightarrow -d^{\ddagger_4}=c_2\im1+c_3\im2-c_6\jh1$.\footnote{For any $d\in\TC$, the symbol $d^{\ddagger_4}$ denotes its fourth tricomplex conjugate \cite{GarantRochon}.}  
Since $\suite$ is bounded if and only if $\{P_{-d^{\ddagger_4}}^{n}(0)\}_{n\in\N}$ and $\{P_{d}^{n}(0)\}_{n\in\N}$ are bounded, we have 
\begin{align*}
c\in\T(\im1,\im2,\jh2) &\Leftrightarrow -d^{\ddagger_4}\in \T(\im1,\im2,\jh1)\text{ and }d\in\T(\im1,\im2,\jh1) \\
&\Leftrightarrow d\in\T(\im1,\im2,-\jh1)\text{ and }d\in \T(\im1,\im2,\jh1) \\
&\Leftrightarrow d\in\T(\im1,\im2,-\jh1)\mcap\T(\im1,\im2,\jh1).
\end{align*}
Observe that $\T(\im1,\im2,\pm\jh1)\not\subset\Tbb(\im1,\im2,\jh2)$. In an analogous manner to that of proposition \ref{Firebrot characterization}, we define
\[
\Ta(\im1,\im2,\jh1) := \{c_2\im1+c_3\im2+c_6\jh2:c_2\im1+c_3\im2+c_6\jh1 \in \T(\im1,\im2,\jh1)\}
\]
to make sure the 3D sets involved are in the right subspace. Thus,
\begin{align*}
\T(\im1,\im2,\jh2) &= \{c\in\Tbb(\im1,\im2,\jh2) : \suite \text{ is bounded}\}\\
&= \begin{multlined}[t]
\{c=c_2\im1+c_3\im2+c_6\jh2:\\
c_2\im1+c_3\im2+c_6\jh1\in\T(\im1,\im2,-\jh1)\mcap\T(\im1,\im2,\jh1) \}
\end{multlined}\\
&=\{c\in\Tbb(\im1,\im2,\jh2):c\in\Ta(\im1,\im2,-\jh1)\mcap\Ta(\im1,\im2,\jh1) \}. \qedhere
\end{align*}
\end{proof}
The above result reveals that the Turtlebrot $\T(\im1,\im2,\jh2)$ can be expressed as the intersection of two other 3D slices. Moreover, in the right subspace, one of these slices is none other than a duplicate of the Mousebrot (Fig.\ref{fig:mousebrot}), while the second is obtained by applying a reflection along the plane $z\jh2=0$ to the first.\footnote{The reference \cite{Vallieres} contains another characterization of the same principal slice involving the Tetrabrot $\T(1,\im1,\im2)$ with all the details.} Interestingly, only one other principal 3D slice possesses the same property: the Hourglassbrot.
\begin{proposition}\label{prop:hourglassbrot}
Consider the principal 3D slice $\T(\im1,\jh1,\jh2)$. We have
\begin{gather*}
\T(\im1,\jh1,\jh2) = \Ta(1,\im1,\jh1)\mcap\Ta(-1,\im1,\jh1)
\shortintertext{where}
\Ta(1,\im1,\jh1) := \{c_2\im1+c_4\jh1+c_6\jh2:c_2\im1+c_4\jh1+c_6 \in \T(1,\im1,\jh1)\}.
\end{gather*}
\end{proposition}
\begin{proof}
We begin by setting $c=c_2\im1+c_4\jh1+c_6\jh2\in\T(\im1,\jh1,\jh2)$ and then using identity \eqref{eq:idemp de}. The rest of the proof is similar to that of Proposition \ref{prop:turtlebrot}.
\end{proof}
Thus, in the right subspace, the Hourglassbrot $\T(\im1,\jh1,\jh2)$ can be expressed as the intersection of two 3D slices: a copy of the Arrowheadbrot (Fig.\ref{fig:arrowheadbrot}), and the slice obtained by applying a reflection along the plane $z\jh2=0$ to it.

\section{The cube and the stellated octahedron}\label{sec:cube}
The main reason why the principal slices $\T(1,\jh1,\jh2)$ and $\T(\jh1,\jh2,\jh3)$ are Platonic solids is that in both cases, the iterates calculated when applying the Mandelbrot algorithm to numbers of the form $c_1+c_4\jh1+c_6\jh2$ and $c_4\jh1+c_6\jh2+c_7\jh3$ stay in a particular four-dimensional subspace of $\TC$ called the \textit{biduplex numbers} \cite{Vallieres}. Indeed, it is easily verified that the set of biduplex numbers
\[
\D(2) := \{c_1+c_4\jh1+c_6\jh2+c_7\jh3:c_1,c_4,c_6,c_7\in\R\text{ and }\jh1^2=\jh2^2=\jh3^2=1\}
\]
together with tricomplex addition and multiplication forms a proper subring of $\TC$, and is thus closed under these operations. Moreover, in the right idempotent basis \eqref{eq:idemp 4}, every biduplex number can be expressed through four real components. It follows that the boundedness of the sequence $\suite$ specific to any $c\in\T(1,\jh1,\jh2)$ or $c\in\T(\jh1,\jh2,\jh3)$ is entirely dependent on the dynamics of the classical Mandelbrot set $\M1$ along the real axis, the intersection of which corresponds to a single interval, hence the regularity.

This suggests another approach to generate regular polyhedra within tricomplex dynamics: to define and visualize 3D slices in a basis that is directly linked to the simple dynamics of the real line. Since identity \eqref{eq:idemp 4} indicates that the set $\{\ut,\uct,\utc,\uctc\}$ is a basis of the vector space of $\TC$ with complex coefficients, the set $\{\ut,\uct,\utc,\uctc,\im1\ut,\im1\uct,\im1\utc,\im1\uctc\}$ is a basis of the same space, but with real coefficients. This brings us to the following definitions.
\begin{definition}
Let $\al,\b,\de$ be three distinct elements taken in $\{\ut,\uct,\utc,\allowbreak\uctc,\im1\ut,\im1\uct,\im1\utc,\im1\uctc\}$. The space
\[
\mT(\al,\b,\de) := \vspan\{\al,\b,\de\}
\]
is the vector subspace of $\TC$ consisting of all real finite linear combinations of these three distinct units.
\end{definition}
\begin{definition}
Let $\al,\b,\de$ be three distinct elements taken in $\{\ut,\uct,\utc,\allowbreak\uctc,\im1\ut,\im1\uct,\im1\utc,\im1\uctc\}$. We define an idempotent 3D slice of the tricomplex Mandelbrot set $\M3$ as
\begin{align*}
\Te(\al,\b,\de) &= \{c\in\mT(\al,\b,\de) : \suite\text{ is bounded}\} \\
&= \mT(\al,\b,\de)\cap\M{3}.
\end{align*}
\end{definition}
Although there are still fifty-six possible slices in total, it is obvious that few distinct 3D dynamics actually occur in this basis. For the sake of brevity, we will introduce the only 3D slice that is of interest regarding our objective. Figure \ref{fig:earthbrot} illustrates the idempotent 3D slice $\T_{e}(\ut,\uct,\utc)$, called the \textit{Earthbrot}.
\begin{figure}[h]
\centering
\includegraphics[width=7cm]{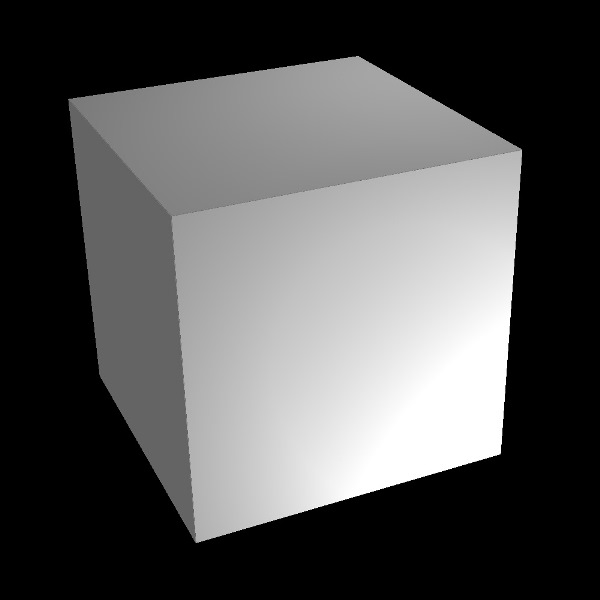}
\caption{The idempotent slice $\T_{e}(\ut,\uct,\utc)$.}
\label{fig:earthbrot}
\end{figure}
\begin{proposition}
The Earthbrot is a cube with edge length of $\frac{9}{4}$.
\end{proposition}
\begin{proof}
Take $x=x_1\ut+x_2\uct+x_3\utc\in\mT(\ut,\uct,\utc)$ such that $\{P_{x}^{n}(0)\}_{n\in\N}$ is bounded. Then, the properties of the idempotent representation \eqref{eq:idemp 4} imply this is the case if and only if the sequences $\{P_{x_{i}}^{n}(0)\}_{n\in\N}$ are bounded, and since $x_i\in\R$, we must have $x_i\in\left[-2,\,\frac{1}{4}\right], i=1,2,3$. Thus, in the basis $\{\ut,\uct,\utc\}$, we obtain the equality 
\[
\T_{e}(\ut,\uct,\utc) = \left[-2,\,\frac{1}{4}\right] \times \left[-2,\,\frac{1}{4}\right] \times \left[-2,\,\frac{1}{4}\right],
\]
where $\times$ denotes the standard cartesian product.
\end{proof}
It seems highly unlikely that the two remaining Platonic solids, the dodecahedron and the icosahedron, can be visualized through tricomplex dynamics. However, it is possible to generate other types of polyhedra, like regular compounds and some specific Archimedean solids. To conclude this section, we wish to give a basic example of the first type. The stellated octahedron, can be seen as a regular dual compound made of two regular tetrahedra. As such, it can be regarded as the simplest polyhedral compound. In our context, Theorem \ref{thm:tetraedre} states that the principal 3D slice $\T(\jh1,\jh2,\jh3)$ is a regular tetrahedron, while several proofs in section \ref{sec:geometrical} hinted that tricomplex conjuation can have a fundamental geometric sense in specific 3D subspaces. More precisely, its effect can be equivalent to that of a reflection along a certain plane. Combining these ideas gives us a simple way, within tricomplex dynamics, to generate a stellated octahedron as the union of the slice $\T(\jh1,\jh2,\jh3)$ and of its geometric dual.

By proposition \ref{Firebrot characterization}, it is not too hard to see that the latter can be obtained by applying a reflection along the plane $y\jh2=0$ to the Firebrot. Moreover, since $c=c_{4}\jh1 + c_{6}\jh2 + c_{7}\jh3\,\Leftrightarrow\,-(c^{\ddagger_{5}})=c_{4}\jh1 - c_{6}\jh2 + c_{7}\jh3$, we deduce that the operation $-(c^{\ddagger_{5}})$ corresponds to the desired reflection. Therefore, the 3D slice
\[
\T(\jh1,-\jh2,\jh3):=\{-(c^{\ddagger_{5}})\,|\,c\in\T(\jh1,\jh2,\jh3)\}
\]
must coincide with the geometric dual of the slice $\T(\jh1,\jh2,\jh3)$. Generating these simultaneously results in the polyhedron illustrated in Figure \ref{fig:starbrot}.
\begin{figure}[ht]
\centering
\includegraphics[width=7cm]{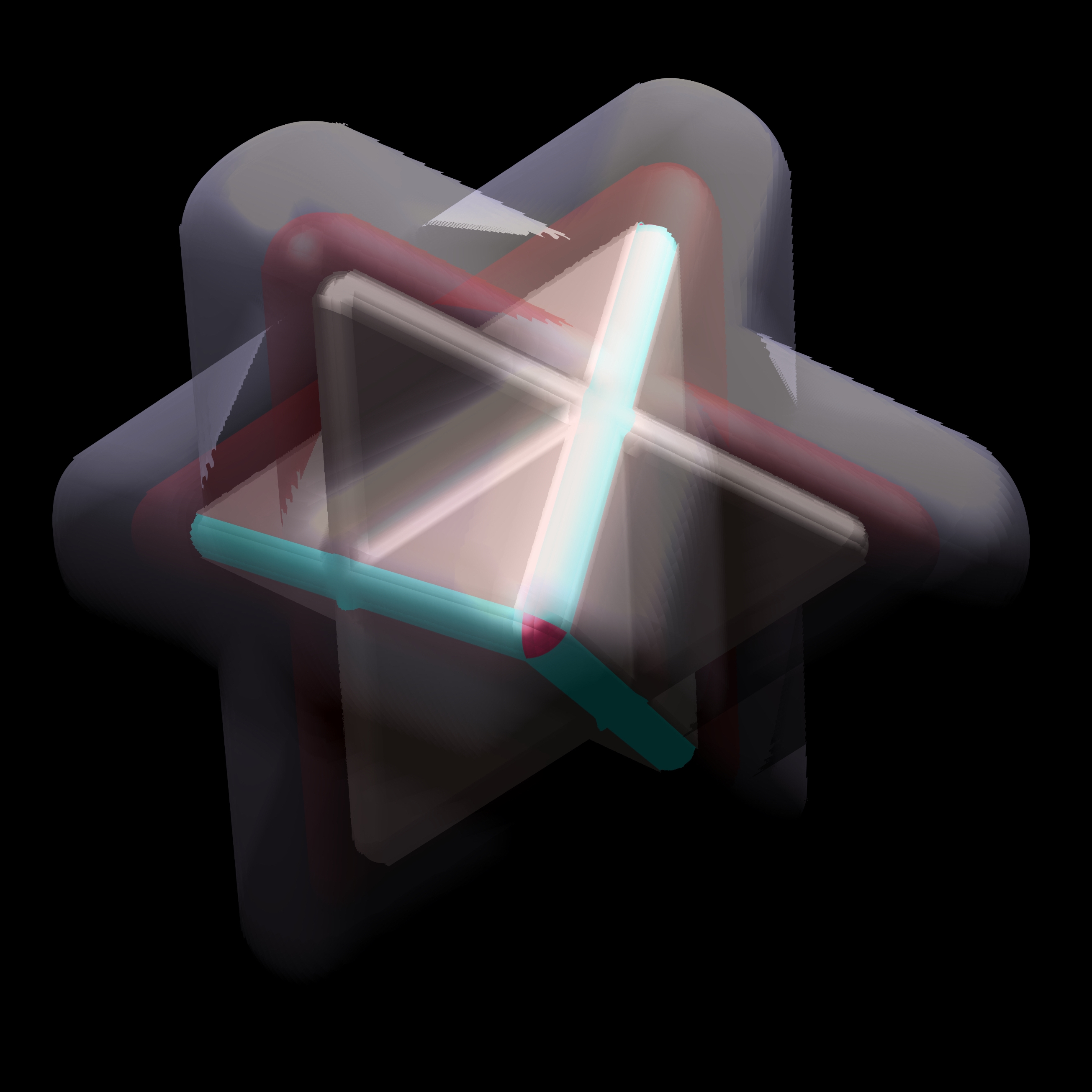}
\caption{The stellated octahedron (also called Starbrot) with various divergence layers.}
\label{fig:starbrot}
\end{figure}
%
%
%\caption{The Starbrot with various divergence layers}
%
%corresponding to the regular dual compound $\T(\jh1,\jh2,\jh3)\cup\T(\jh1,-\jh2,\jh3)$
%
%

%
\section*{Conclusion}
In this article, we first established new results in the algebra of tricomplex numbers related to idempotent elements and invertibility, thus paving the way for interesting extensions valid in the multicomplex setting $\Mu{n},n\geq3$. Then, we presented various geometrical characterizations for the principal 3D slices of the tricomplex Mandelbrot set $\M3$, allowing these to be classified according to their connections to 2D or 3D related sets. In the process, we also confirmed the presence of three Platonic solids within the tricomplex dynamics associated with the Mandelbrot algorithm.

In subsequent works, it could be interesting to extend the classification of the principal slices of $\M3$ to the power $p\geq2$. Doing so would probably expand the list of convex polyhedra found among the principal slices because when considering $p=8$, the Firebrot strongly resembles a truncated tetrahedron, which is an Archimedean solid (Figure \ref{fig:archibrot}).
\begin{figure}[h]
\centering
\includegraphics[width=7cm]{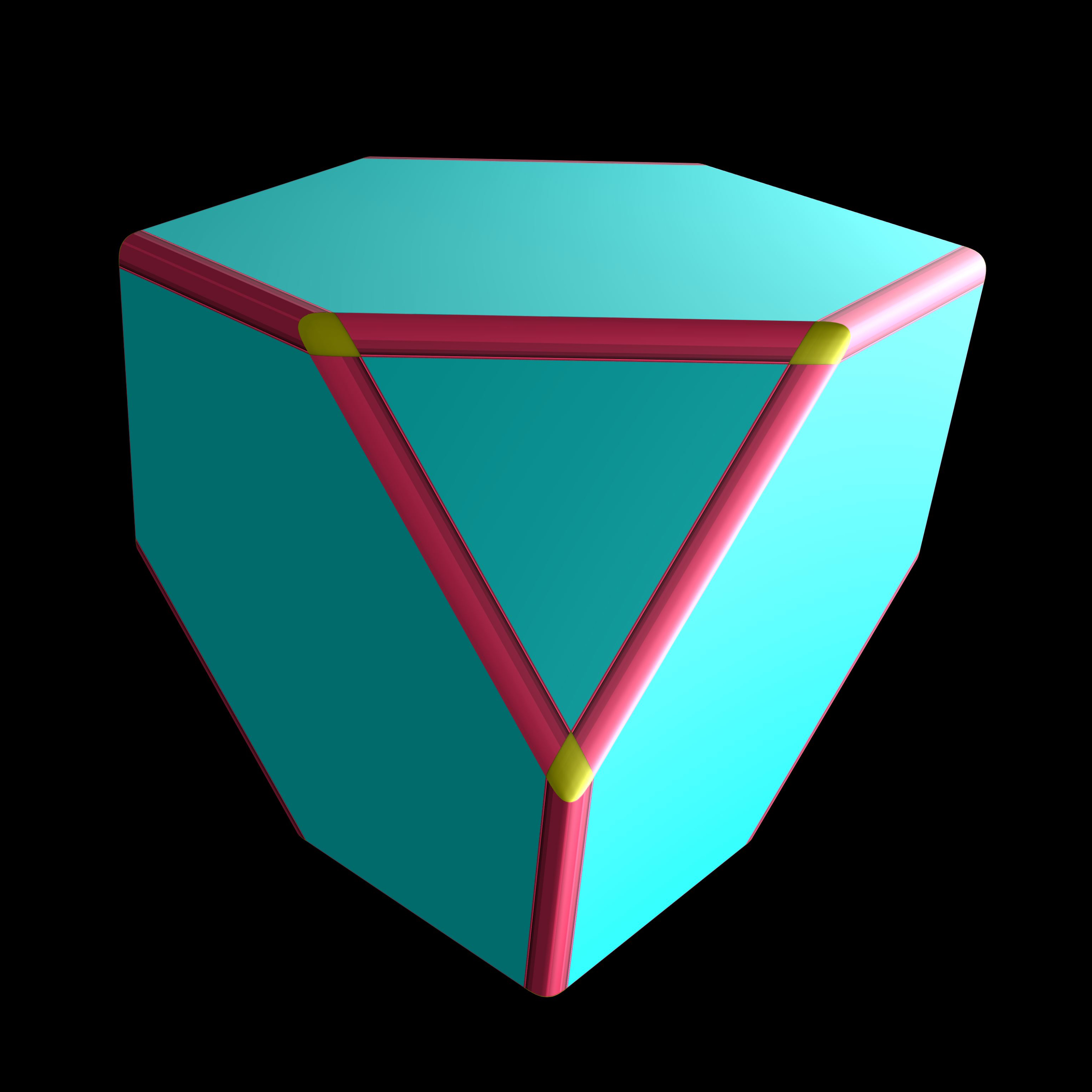}
\caption{The principal 3D slice $\T(\jh1,\jh2,\jh3)$ power $8$.}
\label{fig:archibrot}
\end{figure}

In addition, by considering the algebra $\Mu{n},n\geq3$, the search for regular convex polytopes could be generalised to $n$-dimensional slices. In fact, it is worth noting that the method used in section \ref{sec:cube} provides a straightforward way to establish that the idempotent 4D slice $\T_{e}(\ut,\uct,\utc,\uctc)$ is a tesseract (also called hypercube), that is, a four-dimensional regular convex polytope. Furthermore, the approach in Theorem \ref{thm:tetraedre} can probably be used to prove that in the usual basis, at least one specific 4D slice corresponds to a regular four-dimensional cross-polytope (also called hyperoctahedron). Together, these examples provide sufficient information to allow us to emit a conjecture.
\begin{conjecture}
Let $n\geq3$. The multicomplex Mandelbrot set $\M{n}$ contains exactly three regular convex $n$-polytopes among all possible principal or idempotent $n$-dimensional slices.
\end{conjecture}
In view of what has been mentioned above, it appears the easiest way to start to prove this conjecture is to show that both the $n$-dimensional hypercube and the $n$-dimensional hyperoctahedron exist within multicomplex dynamics when $n\geq3$. This is supported by the fact that the number of equations needed to describe these $n$-polytopes elegantly match the number of equations provided by the extended representation \eqref{eq:idemp 4} under any of the two bases considered herein. However, note that this is not the case for the $n$-dimensional simplex. Thus, proving its existence for every value of $n\geq3$ promises to be a more arduous task. Finally, in a more applied perspective, it could be relevant to consider this theory in relation with the natural geometry of fragmentation \cite{Fragmentation}.

\section*{Acknowledgements}
 
DR is grateful to the Natural Sciences and Engineering Research Council of Canada (NSERC) for the financial support of this project. AV would like to thank the FRQNT and the ISM for the awards of graduate research grants. The authors are grateful to Louis Hamel and Étienne Beaulac, from UQTR, for their useful work on the MetatronBrot Explorer (MBE8D) in Java.

%\nocite{*}% Dé-commenter pour afficher l'ensemble des références du fichier .bib, même celles non citées
\begin{multicols}{2}
\bibliographystyle{abbrv}
\bibliography{ArXiv_Final}
\end{multicols}

\end{document}